\documentclass[12pt,reqno]{amsart}
\usepackage{amsmath}
\usepackage{amssymb}
\usepackage{latexsym}
\usepackage{float}
\usepackage{amsfonts,setspace}
\usepackage{fullpage}
\usepackage{indentfirst}
\usepackage[]{graphicx,color,tikz}
\usepackage{hyperref}
\usepackage{algorithm,algpseudocode}
\usepackage{color}
\usepackage{changepage}
\usepackage{bm}
\usepackage[numbers,sort]{natbib}

\usepackage[T1]{fontenc}

\usepackage{todonotes}

\usepackage{makecell}
\usepackage{pict2e}
\usepackage{amsthm}
\theoremstyle{plain}
\newtheorem{thm}{Theorem}[section]
\newtheorem{rmk}[thm]{Remark}
\newtheorem{prop}[thm]{Proposition}
\newtheorem{cor}[thm]{Corollary}
\newtheorem{lem}[thm]{Lemma}
\newtheorem{conj}[thm]{Conjecture}
\newtheorem*{thm*}{Theorem}
\newtheorem*{prop*}{Proposition}
\newtheorem*{cor*}{Corollary}

\theoremstyle{definition}
\newtheorem{defn}[thm]{Definition}
\newtheorem{exm}[thm]{Example}
\newtheorem{ex}[thm]{Example}

\newtheorem{ques}[thm]{Question}

\newcommand{\rr}{\mathbb{R}}

\newcommand{\cc}{\mathbb{C}}

\newcommand{\kk}{\mathbb{K}}

\newcommand{\bfx}{\mathbf{x}}
\newcommand{\bfy}{\mathbf{y}}

\DeclareMathOperator{\esd}{esd}
\DeclareMathOperator{\rank}{rank}
\DeclareMathOperator{\In}{In}
\DeclareMathOperator{\diag}{diag}

\DeclareMathOperator{\gcr}{gcr}
\DeclareMathOperator{\sign}{sign}

\newcommand{\kl}[1]{{\color{purple} KL: #1}}

\title{Typical ranks in symmetric matrix completion}
\author{Daniel Irving Bernstein}
\address{Institute for Data, Systems, and Society, Massachusetts Institute of Technology, Cambridge, MA 02139}
\email{dibernst@mit.edu}\urladdr{https://dibernstein.github.io}
\author{Grigoriy Blekherman}
\address{School of Mathematics, Georgia Institute of Technology, Atlanta, GA 30332}
\email{greg@math.gatech.edu}\urladdr{https://sites.google.com/site/grrigg/}
\author{Kisun Lee}
\address{School of Mathematics, Georgia Institute of Technology, Atlanta, GA 30332}
\email{klee669@gatech.edu}\urladdr{https://people.math.gatech.edu/~klee669/}
\begin{document}

\begin{abstract} We study the problem of low-rank matrix completion for symmetric matrices.
	The minimum rank of a completion of a generic partially specified symmetric matrix
	depends only on the location of the specified entries,
	and not their values, if complex entries are allowed.
	When the entries are required to be real,
	this is no longer the case and the possible minimum ranks are called \emph{typical ranks}.
	We give a combinatorial description of the patterns of specified entries of $n\times n$ symmetric matrices
	that have $n$ as a typical rank.
	Moreover, we describe exactly when such a generic partial matrix is minimally completable to rank $n$.
	We also characterize the typical ranks for patterns of entries with low maximal typical rank.
\end{abstract}

\maketitle


\section{Introduction}
This paper is concerned with the \textit{symmetric low-rank matrix completion problem}.
We begin with an illustrative example.
Suppose that we have a partially specified symmetric matrix $\begin{pmatrix}a&*\\ *&b \end{pmatrix}$, where $a$ and $b$ are given, and our objective is to find the unknown entry $*$, so that the full matrix has minimal rank. Unless $a$ and $b$ are both zero, any completion will have rank at least $1$, and if we are allowed complex entries, then can always complete to rank $1$ by setting $*=\sqrt{ab}$. This situation is quite general: if we fix a pattern of known and unknown entries and the entries are complex numbers, then outside of a low-dimensional subset in the space of partial fillings (the point $(0,0)$ in the $(a,b)$-space in our example), any partial filled matrix can be completed to the same minimal rank, called the \textit{generic completion rank} of the pattern. We note that in general the exceptional low-dimensional set will contain matrices that are minimally completable to ranks that are both higher and lower than the generic completion rank.

If we consider our example when entries are restricted to be real numbers, then the situation is more complicated.
If $ab>0$,
then we can still complete to rank one,
but if $ab <0$, then we can only complete to rank two.
Notice that the set of matrices that are completable to rank two forms a full-dimensional subset of the $(a,b)$-space of partial fillings. This brings us to a crucial definition: given a fixed pattern of known and unknown entries,
a rank $r$ is called \emph{typical} if the set of all matrices with real entries minimally completable to rank $r$ forms a full-dimensional subset of the vector space of partial fillings.
As we see in the above example, a given pattern can have more than one typical rank.

It is known that the generic completion rank of a pattern is equal to its lowest typical rank,
and all ranks between the maximal typical rank and the minimal typical rank are also typical
\cite{bernstein2018typical, Bernardi2018}.
We say that a pattern of known entries of an $n\times n$ partial matrix is \emph{full-rank typical} if $n$ is a typical rank.
The question of characterizing full-rank typical patterns was raised in \cite{bernstein2018typical}.
One of our main results, Theorem \ref{thm:fullranktypicalCharacterization}.
is a simple characterization of the full-rank typical patterns.
We provide a semialgebraic description of the set of generic partial matrices
that can only be completed to full rank, in the case that the pattern of known entries is full-rank typical
(Theorem \ref{thm:whenFullRank}),
and for one particular family of patterns, we give a semialgebraic description
of the open regions corresponding to each typical rank (Theorem \ref{thm:esdArgument-disjointcliques}).
We also characterize the typical rank behavior of patterns with generic completion rank one (Theorem \ref{thm:typicalrank2Graphs}),
and of patterns with generic completion rank two such that all diagonal entries are known (Theorem \ref{thm:loopedTrees}).

Generic completion rank for symmetric matrices has applications in statistics as a bound for the \textit{maximum likelihood threshold} of a Gaussian graphical model and in \textit{factor analysis} \cite{uhler2012geometry, GSMR3706762,MR3903791, shapiro1982rank}. If we restrict to \textit{positive semidefinite completions}, then maximal typical rank of a pattern (suitably defined) is known as the \textit{Gram dimension}, and is closely related to Euclidean distance realization problems \cite{MR3207690,MR3201103}. We note that for the positive semidefinite matrix completion, there are no partial matrices that can be completed only to full rank as any entry of a positive definite matrix may be changed to make the matrix positive semidefinite and drop rank. There is also a similarity to the investigation of generic and typical ranks for tensors and symmetric tensors \cite{Bernardi2018,BTMR3368091,MR3633774,MR2865915}. We now state and discuss our main results in detail.

\subsection{Main Results in Detail}
Matrices and partial matrices will be assumed to have entries in a field $\kk$,
which will always be $\rr$ or $\cc$.
Let $\mathcal{S}^n(\kk)$ denote the set of $n\times n$ symmetric matrices with entries in $\kk$.
To a pattern of known entries, we associate a semisimple graph $G=([n],E)$
(i.e.~loops are allowed, but no multiple edges),
where the edges of $G$ correspond to the known entries
and non-edges of $G$ correspond to the unknown entries (See Figure \ref{fig:partialmatrix}).
Associated to each semisimple graph $G$ is the set of \emph{G-partial matrices},
which are elements of $\kk^{E}$.
It is often helpful to think of a completion of a $G$-partial matrix $M$ as a function
\[
	M: \kk^{[n]\sqcup \binom{[n]}{2}\setminus E} \rightarrow \mathcal{S}^n(\kk)
\]
that simply plugs in a set of values for the missing entries.
Thus given a partial matrix $M$,
we let $M(\bfx)$ denote the matrix obtained by plugging in $\bfx$ for the missing entries of $M$.
 	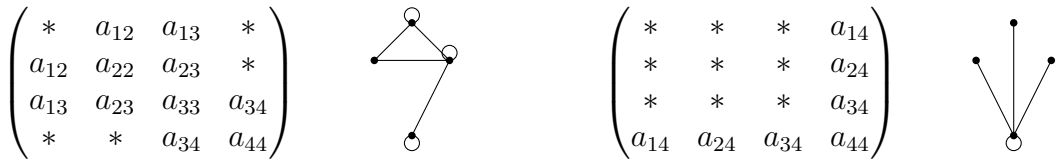
\begin{figure}[H]
	\begin{tikzpicture}[roundnode/.style={circle, fill=black, inner sep=0pt, minimum size=1mm}]
	
	
	\node[roundnode](A) at (5,1.6) {};
	\node[roundnode](B) at (5.5,2.1) {};
	\node[roundnode](A) at (6,1.6) {};
	\draw (5.5,2.2) circle [thick,radius = .1];
	\draw (6,1.7) circle [thick,radius = .1];
	
	\node[roundnode](B) at (5.5,0.6) {};
	\draw (5.5,0.5) circle [thick,radius = .1];
	
	\draw (5,1.6) -- (5.5, 2.1) {};
	\draw (6,1.6) -- (5.5, 2.1) {};
	\draw (6,1.6) -- (5, 1.6) {};
	\draw (5.5,.6) -- (6, 1.6) {};

    \node at (2,1.3) {$\begin{pmatrix}
    	* & a_{12} & a_{13} &*\\
		a_{12} & a_{22} & a_{23} &*\\
		a_{13} & a_{23} & a_{33} & a_{34}\\
		* & * & a_{34} & a_{44}
    	\end{pmatrix}$};

	 	 2nd graph
	
	 	 4 vertices in the middle
	 	\node[roundnode](A) at (13,1.6) {};
	 	\node[roundnode](A) at (13.5,2.1) {};
	 	\node[roundnode](B) at (14,1.6) {};
	\draw (13.5,0.5) circle [thick,radius = .1];
	
	 	 2 vertices in the end sides
	 	\node[roundnode](B) at (13.5,0.6) {};

	 	 edges from the vertex above
	 	\draw (13.5,0.6) -- (13, 1.6) {};
	 	\draw (13.5,0.6) -- (13.5, 2.1) {};
	 	\draw (13.5,0.6) -- (14, 1.6) {};
    \node at (10,1.3) {$\begin{pmatrix}
	* & * & * & a_{14}\\
	* & * & * & a_{24}\\
	* & * & * & a_{34}\\
	a_{14} & a_{24} & a_{34} & a_{44}
	\end{pmatrix}$};

	\end{tikzpicture}
	\caption{Partial matrices alongside their corresponding graphs.}
	\label{fig:partialmatrix}
\end{figure}

We will use the term \textit{generic matrix} to mean any matrix that lies outside of an (often unspecified) algebraic subset of the space of all matrices. For instance, we can say that a generic matrix is invertible, since non-invertible matrices lie inside the determinant hypersurface. Similarly, a generic square matrix has distinct eigenvalues, since discriminant of the characteristic polynomial catches all matrices with repeated eigenvalues. Specifying the algebraic subset explicitly is often omitted.

For any graph $G$ there exists an integer $r$ such that
for any generic $G$-partial matrix $M$ with complex entries,
there exists a complex $\bfx$ such that $M(\bfx)$ has rank $r$, and $M$ cannot be completed to a rank below $r$ \cite[Proposition 6.1(1)]{bernstein2018typical}.
This $r$ is called the \emph{generic completion rank of $G$} and we denote it $\gcr(G)$.
If we insist that $\bfx$ be real, then
we lose the existence of generic completion rank
and instead get \emph{typical ranks}.
More precisely,
a typical rank of a graph $G$
is an integer $r$ such that there exists an open set $U \subset \rr^E$ 
in the Euclidean topology of real $G$-partial matrices
such that any $M \in U$ is completable to rank $r$, and cannot be completed to a rank below $r$. In this case we
say that $M$ is \emph{minimally completable to rank $r$}.

\begin{prop*}[c.f. {\cite[Proposition 6.1]{bernstein2018typical}}] \label{prop:typicalrankProperties}
    Let $G = ([n],E)$ be a semisimple graph.
    The minimum typical rank of $G$ is $\gcr(G)$
    and all integers between $\gcr(G)$ and the maximum typical rank of $G$ are typical ranks of $G$. In addition, the maximum typical rank of $G$ is at most $2 \cdot \gcr(G)$.
\end{prop*}


Denote the $n$-clique with a loop at every vertex by $K_n^\circ$.
The \emph{complement} of a semisimple graph $G = ([n],E)$, denoted $G^c$,
is the graph obtained by removing the edges in $E$ from $K_n^\circ$.
A semisimple graph with $n$ vertices is called \textit{full-rank typical} if $n$ is a typical rank.
Our first main result characterizes the full-rank typical graphs, thus solving a problem posed in \cite{bernstein2018typical}. Note that a bipartite semisimple graph cannot have loops.

\begin{thm*}[Theorem \ref{thm:fullranktypicalCharacterization}]
		A graph $G$ is full-rank typical if and only if its complement $G^c$ is bipartite.
\end{thm*}

For any full-rank typical graph $G$,
Theorem  \ref{thm:whenFullRank} describes the set of generic partial matrices that are minimally completable to full rank.

Given graphs $G$ and $H$, we let $G \sqcup H$ denote the disjoint union of $G$ and $H$.
For the full-rank typical graphs $G=K_n^\circ\sqcup K_m^\circ$, we describe how to calculate the minimal completion rank of a generic $G$-partial matrix.
To state this theorem, we need the following definition.

\begin{defn}
For a real full-rank symmetric matrix $A$,
	let $p_A$ and $n_A$ denote the number of positive and negative eigenvalues of $A$ respectively.
	Given two real full-rank symmetric matrices $A$ and $B$ potentially of different sizes, 
	we define \emph{eigenvalue sign disagreement} between $A$ and $B$, denoted $\esd(A,B)$, as follows
	\[	
		\text{esd}(A,B)
		:=\left\{\begin{array}{ll}
		0 & \text{if }(p_A-p_B)(n_A-n_B)\geq 0\\
		\min\{|p_A-p_B|,|n_A-n_B|\} & \text{otherwise}
		\end{array}\right.
	\]
\end{defn}


\begin{thm*}[Theorem \ref{thm:esdArgument-disjointcliques}]
	Let $m,n \ge 0$ be integers
	and $G = K_m^\circ \sqcup K_n^\circ$.
	Let $A$ be a full-rank $m\times m$ symmetric matrix and $B$ be a full-rank $n\times n$ symmetric matrix,
	and consider the following $G$-partial matrix
	\[
	M=\begin{pmatrix}
	A & *\\
	* &B
	\end{pmatrix}.
	\]
	Then, $M$ is minimally completable to rank $\max\{n,m\}+\esd(A,B)$.
\end{thm*}


Our remaining results concern graphs with low typical ranks.
Semisimple graphs with generic completion rank $1$ were characterized in \cite{singer2010uniqueness}.
We characterize their typical ranks.

\begin{thm*}[Theorem \ref{thm:typicalrank2Graphs}]
	Let $G$ be a semisimple graph with generic completion rank $1$.
	The maximum typical rank of $G$ is $2$ if $G$ has at least two cycles, and $1$ otherwise.
\end{thm*}

A tree with at most one non-leaf vertex is called a \textit{star tree}.
A graph is called \textit{looped} if every vertex has a loop.
Theorem 2.5 in \cite{GSMR3706762} implies that a looped graph has generic completion rank at most $2$
if and only if it has no cycles (aside from loops).
We build on this, characterizing the typical ranks of looped graphs with generic completion rank at most~$2$.

\begin{thm*}[Theorem \ref{thm:loopedTrees}]
	Let $G$ be a looped graph.
	\begin{enumerate}
		\item The generic completion rank of $G$ is at most $2$ if and only if $G$ is a looped forest.
		Equality is attained if and only if $G$ has at least one non-loop edge \cite[Theorem 2.5]{GSMR3706762}.
		\item When $G$ has generic completion rank $1$ and at least two vertices, $G$ also has $2$ as a typical rank.
		\item When $G$ has generic completion rank $2$,
		the maximum typical rank of $G$ is
		\begin{enumerate}
			\item $2$ if $G$ has exactly two vertices
			\item $3$ if $G$ has at least three vertices 
			and is the union of a looped star tree and a looped set of isolated vertices, and
			\item $4$ otherwise.
		\end{enumerate}
	\end{enumerate}
\end{thm*}

\section{Full-rank typical graphs}
We first answer a question posed in \cite{bernstein2018typical},
characterizing the graphs that are full-rank typical.
We start with a simple, but important, observation.

\begin{rmk}\label{rmk:addingEdges}
    If $G$ is full-rank typical,
    then any graph obtained by adding edges to $G$ is also full-rank typical.
\end{rmk}

\subsection{The characterization} We now state the main result of this subsection.
Note that a bipartite semisimple graph cannot have loops.

\begin{thm}\label{thm:fullranktypicalCharacterization}
	A graph $G$ is full-rank typical if and only if its complement $G^c$ is bipartite.
\end{thm}
\begin{proof}
	We first show that if $G^c$ is not bipartite, then $G$ is not full-rank typical.
	Remark \ref{rmk:addingEdges} implies that
	removing edges from a graph that is not full-rank typical
	produces another graph that is not full-rank typical.
	Since every non-bipartite graph contains an odd cycle,
	it suffices to let $G$ be a graph
	whose complement consists of an odd cycle and an independent set of vertices,
	and then show that $G$ is not full-rank typical.
	So let $M$ be a generic $G$-partial matrix.
	Then $\det(M(\bfx))$ is a polynomial in the indeterminates $\bfx$, with odd total degree,
	and thus has a real zero.
	So $G$ is not full-rank typical.

	Conversely, if $G^c$ is bipartite, then for some positive integers $m,n$, $G$ contains $K_n^\circ \sqcup K_m^\circ$ as a subgraph. Therefore, $G$ is full-rank typical by Remark \ref{rmk:addingEdges} and Proposition \ref{prop:KnKmfullrank} below.
\end{proof}

\begin{cor}\label{cor:maxTypicalLowerBound}
    The maximum typical rank of $G$ is at least the maximum number of vertices
    in a bipartite induced subgraph of the complement $G^c$.
\end{cor}

One might ask whether the bound given by Corollary \ref{cor:maxTypicalLowerBound} is sharp.
Unfortunately, this is not the case as shown by the following example. 

\begin{exm}
	Consider a looped complete bipartite graph $K^\circ_{m,n}$ for $m\geq 2$ and $n=\binom{m}{2}$. It is known that the generic completion rank of $K^\circ_{m,n}$ is equal to $m$ in \cite[Theorem 2.5]{MR3903791}. Note that the maximum size of a bipartite induced subgraph of $(K^{\circ}_{m,n})^c$ is $4$. Therefore, if we choose $m>4$, then its maximum typical rank is greater than $4$.
\end{exm}

\subsection{The disjoint union of two cliques}
The main result of this subsection is Theorem~\ref{thm:esdArgument-disjointcliques}
which explains how to determine the minimum rank of a completion
of a given generic $G$-partial matrix when $G$ is the disjoint union of two cliques.
We begin with a special case of Theorem \ref{thm:esdArgument-disjointcliques}
that will be necessary for its proof.

\begin{prop}\label{prop:KnKmfullrank}
	Let $m,n \ge 0$ be integers
	and let $G = K_m^\circ \sqcup K_n^\circ$.
	A $G$-partial matrix
	\[
		M=\begin{pmatrix}
		A & *\\
		* & B
		\end{pmatrix}
	\]
	is minimally completable to full-rank if and only if $A$ is positive definite (negative definite resp.) and $B$ is negative definite (positive definite resp.).
	In particular, $G = K_m^\circ \sqcup K_n^\circ$ is full-rank typical.
\end{prop}
\begin{proof}
	Throughout, we will view $M(\bfx)$ as a $2\times 2$ symmetric block matrix.
	We denote the upper-right block, whose entries are given by $\bfx$, by $X$.
	Without loss of generality,
	let $A$ be an $m\times m$ positive definite matrix and $B$ be an $n\times n$ negative definite matrix.
	Let $y \in \ker M(\bfx)$.
	We will write $y$ as
	\[
		y = \begin{pmatrix}
		    y_1 \\
		    y_2
		\end{pmatrix}
		\qquad \textnormal{where} \qquad
		y_1 \in \rr^m \ \ \textnormal{and} \ \ y_2 \in \rr^n.
	\]
	Since $y \in \ker M(\bfx)$, we have
	$Ay_1 =- Xy_2$ and $X^T y_1 =-By_2$, 
	and therefore $y_1^T A y_1=y_2^T B y_2$.
	Since $A$ is positive definite and $B$ is negative definite,
	this implies that $y_1$ and $y_2$ are both zero vectors
	and so $M(\bfx)$ is full-rank for all $\bfx$.
	
	For the converse, assume that there are eigenvalues $a$ of $A$ and $b$ of $B$
	such that $ab \ge 0$.
	If without loss of generality $A$ is not full rank,
	then we may complete to some $M(\bfx)$ so that the columns of $X^T$ satisfy a relation that the columns
	of $A$ satisfy, thus making $M(\bfx)$ not have full rank.
	So assume $ab > 0$.
	Let $C$ and $D$ be orthogonal matrices such that
	$C^T A C$ and $D^T B D$ are diagonal matrices
	whose nonzero entries are the eigenvalues of $A$ and $B$,
	leading with $a$ and $b$ respectively.
	Treating $\bfx$ as a vector of indeterminates,
	moving from
	\[
		M(\bfx) =\begin{pmatrix}
			A & X\\
			X^T &B
		\end{pmatrix}
		\qquad {\rm to} \qquad
		\begin{pmatrix}
		    C^T & 0 \\
		    0 & D^T
		\end{pmatrix}
		\begin{pmatrix}
			A & X\\
			X^T &B
		\end{pmatrix}
		\begin{pmatrix}
		    C & 0 \\
		    0 & D
		\end{pmatrix}
		= \begin{pmatrix}
			C^TAC & C^TXD\\
			D^TX^TC &D^TBD
		\end{pmatrix}
	\]
	corresponds to a linear change of variables when taking determinants,
	so we may without loss of generality assume that $A$ and $B$ are diagonal matrices
	with $a$ and $b$ as the respective leading entries.
	Consider a completion $M(\bfx)$ of $M$ obtained by setting $X_{11}=\sqrt{a_1b_1}$ and $X_{21}=\cdots =X_{n1}=0$.
	The rank formula for Schur complements gives
	\[
		\rank(M(\bfx))=\rank(A)+\rank(B-X^T A^{-1}X).
	\]
	Note that the entries of the first row of $B-X^T A^{-1} X$ are all zero.
	This means that $\rank(B-X^T A^{-1}X)$ is less than $n$
	and so we may complete $M$ to have non-full rank.
\end{proof}

Before we can state Theorem \ref{thm:esdArgument-disjointcliques},
we need the following definition.

\begin{defn}
    Given a real full-rank symmetric matrix $A$,
	let $p_A,n_A$ denote the number of positive and negative eigenvalues of $A$.
	Given two real full-rank symmetric matrices $A$ and $B$ potentially of different sizes, 
	we define \emph{eigenvalue sign disagreement} between $A$ and $B$, denoted $\esd(A,B)$, as follows
	\[	
		\text{esd}(A,B)
		:=\left\{\begin{array}{ll}
		0 & \text{if }(p_A-p_B)(n_A-n_B)\geq 0\\
		\min\{|p_A-p_B|,|n_A-n_B|\} & \text{otherwise}
		\end{array}\right.
	\]
\end{defn}




\begin{thm}\label{thm:esdArgument-disjointcliques}
	Let $m,n \ge 0$ be integers
	and let $G = K_m^\circ \sqcup K_n^\circ$.
	Let $A$ be a full-rank $m\times m$ symmetric matrix and $B$ be a full-rank $n\times n$ symmetric matrix,
	and consider the following $G$-partial matrix
	\[
		M=\begin{pmatrix}
		A & *\\
		* &B
		\end{pmatrix}.
	\]
	Then, $M$ is minimally completable to rank $\max\{n,m\}+\esd(A,B)$.
\end{thm}
\begin{proof}
	Without loss of generality, assume that $n\geq m$.
	Let $C,D$ be orthogonal matrices such that $C^{T}AC$ and $D^{T}BD$ are diagonal.
	Conjugating the indeterminate matrix $M(\bfx)$ by
	\[
	\begin{pmatrix}
	C & 0 \\
	0 & D
	\end{pmatrix},
	\]
	we may assume without loss of generality that $A=\diag(a_1,\dots,a_n)$
	and $B = \diag(b_1,\dots,b_m)$.
	
	We begin by showing that any completion $M(\bfx)$ has rank at least $n+\esd(A,B)$.
	In the case that $\esd(A,B) = 0$, this is implied by the fact that $A$ is a rank-$n$ submatrix of $M(\bfx)$.
	So assume without loss of generality that $0 < \esd(A,B) = |p_A-p_B|$.
	If $\esd(A,B) = p_B-p_A$, then $M(\bfx)$ has a principal submatrix $M'(\bfx)$ of a $2\times 2$ block form
	whose off-diagonal blocks are all indeterminates, whose upper-left block is an $n_A\times n_A$ diagonal matrix
	whose nonzero entries are the negative diagonals of $A$,
	and whose lower-right block is a $p_B\times p_B$ diagonal matrix whose nonzero entries are the positive diagonals of $B$.
	Proposition \ref{prop:KnKmfullrank} implies that any completion of $M'$ has rank $n_A + p_B$.
	But in this situation, $n_A + p_B = p_A + n_A + (p_B-p_A) = n+\esd(A,B)$.
	If $\esd(A,B) = p_A-p_B$, then $p_A-p_B \le n_B-n_A$
	(note that here we are using that $p_A-p_B$ and $n_A-n_B$ have opposite signs by definition of $\esd$).
	This inequality cannot be strict, since otherwise it would contradict $n = p_A + n_A \ge p_B+n_B = m$.
	So $\esd(A,B) = n_B-n_A$ and so we can proceed just as in the case where $\esd(A,B) = p_B-p_A$.
	
	Now we show that we can complete $M$ to rank $n+\esd(A,B)$.
	Letting $X$ denote the upper-right block of $M(\bfx)$,
	we proceed by choosing $\bfx$ in a way such that
	\[
		\rank(B-X^TA^{-1}X) = \esd(A,B).
	\]
	This suffices because $\rank(M(\bfx)) = \rank(A) + \rank(B-X^TA^{-1}X)$ by the rank formula for Schur complements.
	Let $s$ be the maximum number such that $a_ib_i >0$ for all $i \le s$,
	and assume that the ordering of $(a_1,\dots,a_n)$ and $(b_1,\dots,b_m)$ are
	chosen to maximize $s$.
	Note that $s \le m$ and that $\esd(A,B) = m-s$.
	The $ij$th entry of $X^TA^{-1}X$ is $\sum_{k=1}^n\frac{x_{ki}x_{kj}}{a_k}$.
	Therefore, if we set $x_{kk} = \sqrt{\frac{b_k}{a_k}}$ for $1 \le k \le s$
	and all other $x_{kl} = 0$,
	$B-X^T A^{-1} X$ is a diagonal matrix with precisely $\esd(A,B)$ nonzero entries
	and thus has rank $\esd(A,B)$.
\end{proof}

\begin{cor}\label{cor:typicalRankKnKm}
	The typical ranks of $K_n^\circ\sqcup K_m^\circ$ are $\max\{n,m\},\dots, n+m$.
\end{cor}

\subsection{The space of G-partial matrices}
	In this subsection, we consider the following question: given a full-rank typical graph $G$ and a generic $G$-partial matrix $M$,
	when are all completions of $M$ full-rank?
	Theorem \ref{thm:whenFullRank} gives a complete answer to this question.
	It is more or less a direct consequence of Lemma \ref{lemma:oneMissingEdge} ,
	which handles the case where $G$ is obtained from the complete semisimple graph
	by removing a single non-loop edge.

	Given a (partial) matrix $M$ and subsets $S,T$ of the row and column indices,
	we let $M_{S,T}$ denote the (partial) matrix obtained by \emph{removing} the rows corresponding to the elements of $S$ and the columns corresponding
	to the elements of $T$.


	\begin{lem}\label{lemma:oneMissingEdge}
	    Let $M$ be a real partial symmetric $n\times n$ matrix where the $(1,n)$-entry is the only unknown.
	    Then $M$ can be completed to rank $n-1$ or less if and only if
	    \begin{enumerate}
	    	\item $n = 2$ and $\det(M_{1,1})\det(M_{n,n}) \ge 0$, or
	    	\item $\det(M_{1n,1n})\neq 0$ and $\det(M_{1,1})\det(M_{n,n}) \ge 0$, or
	        \item $\det(M_{1n,1n}) = 0$, and $\det(M(0)_{1,n}) \neq 0$ or $\det(M(0)) = 0$.
	    \end{enumerate}
	\end{lem}

	Before proving Lemma \ref{lemma:oneMissingEdge}, we need a lemma about relations among determinants of arbitrary square matrices.

	\begin{lem}\label{lem:relationAmongDeterminants}
	    Let $A$ be a square, not necessarily symmetric, $n\times n$ matrix. Then
	    \[
	    	\det(A)\det(A_{1n,1n})-\det(A_{1,1})\det(A_{n,n})+\det(A_{1,n})\det(A_{n,1}) = 0.
	    \]
	\end{lem}
	\begin{proof}
		Define $f: \rr^{n\times n}\rightarrow \rr$ by
		\[
			f(A) = \det(A)\det(A_{1n,1n})-\det(A_{1,1})\det(A_{n,n})+\det(A_{1,n})\det(A_{n,1}).
		\]
		Our goal is to show that $f$ is identically zero.
	    Let $B$ be the $n\times(n-2)$ matrix obtained from $A$ by removing the first and last columns.
	    Let $g:\rr^n\times \rr^n\rightarrow \rr$ be the function given by
	    \[
	    	g(x,y) = f\left(\left(x \ B \ y\right)\right)
	    \]
	    where $\left(x \ B \ y\right)$ denotes the matrix obtained by adding $x$ and $y$ as columns to $B$ on either side.
	    We will proceed by showing that $g$ is identically zero.
	    Note that $g$ is bilinear and alternating,
	    so it is enough to show that $g(e_i,e_j) = 0$ where $e_i$ denotes the $i^{\rm th}$ standard basis vector and $i < j$.
	    Writing $g(e_i,e_j)$ out explicitly, we get
	    \[
	    	g(e_i,e_j) = (-1)^{i+j+n}\left(
	    		\det(B_{ij,\emptyset})\det(B_{1n,\emptyset})
	    		-\det(B_{1j,\emptyset})\det(B_{in,\emptyset})
	    		+\det(B_{1i,\emptyset})\det(B_{jn,\emptyset})
	    	\right).
	    \]
	    The above is a Grassmann-Pl\"ucker relation, so it is identically zero \cite[Chapter 4.3]{maclagan2015introduction}.
	\end{proof}

	\begin{proof}[Proof of Lemma \ref{lemma:oneMissingEdge}]
		We write $\det(M(m_{1,n}))$ as a quadratic polynomial in $m_{1,n}$ as follows
		\[
			\det(M(m_{1,n})) = -\det(M_{1n,1n})m_{1,n}^2 + 2\det(M(0)_{1,n})m_{1,n} + \det(M(0)).
		\]
		If $\det(M_{1n,1n}) = 0$, this is a linear or constant polynomial.
		It has a zero, which is real, if and only if $\det(M(0)_{1,n}) \neq 0$ or $\det(M(0)) = 0$.

		If $\det(M_{1n,1n}) \neq 0$, then $\det(M(m_{1,n}))$ has a real zero if and only if its discriminant is nonnegative.
		The discriminant of $\det(M(m_{1,n}))$ is
		\[
			4\det(M(0)_{1,n})^2+4\det(M_{1n,1n})\det(M(0)).
		\]
		Lemma \ref{lem:relationAmongDeterminants} implies that the following polynomial is identically zero
		\[
\det(M(0)_{1,n})^2+\det(M_{1n,1n})\det(M(0))-\det(M_{1,1})\det(M_{n,n}).
\]
		Therefore the discriminant of $\det(M(m_{1,n}))$ is $\det(M_{1,1})\det(M_{n,n})$.
	\end{proof}

\begin{defn}
	Let $A$ be a symmetric matrix.
	Define $p_A,n_A$ be the number, counted with multiplicity,
	of positive and negative eigenvalues of $A$.
    The \emph{inertia} of $A$
    is the vector
    \[
    	\In(A) := (p_A,n_A,\dim \ker A).
    \]
\end{defn}

\begin{prop}\label{prop:full-rankInertia}
	Let $G$ be any full-rank typical semisimple graph, and $M$ be a real $G$-partial matrix.
	If $M$ is minimally completable to full-rank,
	then all completions of $M$ have the same inertia.
\end{prop}
\begin{proof}
	Let $M(\bfx_1)$ and $M(\bfx_2)$ be completions of $M$ such that $\In(M(\bfx_1)) \ne \In(M(\bfx_2))$.
	By continuity of the function sending a matrix to its eigenvalues,
	there exists a point $\bfx_0$ on the line segment from $\bfx_1$ to $\bfx_2$
	such that $M(\bfx_0)$ has a zero eigenvalue,
	i.e.~is rank deficient.
\end{proof}

\begin{defn}
    Let $G$ be a full-rank typical graph and let $M$ be a $G$-partial matrix.
    We define $\sign(M) := 0$ if $\det(M(\bfx)) = 0$ for some choice of real $\bfx$,
    and otherwise, define $\sign(M)$ to be the sign of $\det(M(0))$ (which is the sign of $\det(M(\bfx))$ for any $\bfx$
    by Proposition \ref{prop:full-rankInertia}).
\end{defn}

\begin{rmk}\label{rmk:vertexDeletion}
    Let $G$ be a full-rank typical graph.
    Then it follows from Theorem~\ref{thm:fullranktypicalCharacterization}
    that any subgraph obtained from $G$ by deleting vertices is also full rank typical.
    In particular, if $M$ is a $G$-partial matrix and $A$ is a principal minor of $M$,
    then $\sign(A)$ is well-defined.
\end{rmk}

\begin{lem}\label{lem:recursionToTestFullRank}
    Let $G = ([n],E)$ be a full-rank typical graph, let $M$ be a generic $G$-partial matrix,
    and let $\{i,j\}$ be a non-edge of $G$.
    Then $M(\bfx)$ is full-rank for all $\bfx$ if and only if
    $n=2$ or $\sign(M_{ij,ij}) \neq 0$, and $\sign(M_{i,i})$ and $\sign(M_{j,j})$ are nonzero and opposite.
\end{lem}
\begin{proof}
	Without loss of generality, let $\{i,j\} = \{1,n\}$.
	Assume $n = 2$ or $\sign(M_{1n,1n}) \neq 0$.
    If one of $\sign(M_{1,1})$ or $\sign(M_{n,n})$ is zero,
    then Lemma~\ref{lemma:oneMissingEdge} implies that $M$ can be completed to rank $n-1$ or less.
    If not, then the values of $\sign(M_{1,1})$ and $\sign(M_{n,n})$
    do not depend on how we complete the non-$(1,n)$ entries.
    Lemma~\ref{lemma:oneMissingEdge} then implies that $M$ can be completed
    to rank $n-1$ or less if and only if $\sign(M_{1,1})=\sign(M_{n,n})$.

	Now assume $n \ge 3$ and $\sign(M_{1n,1n}) = 0$.
	We show that $M$ can be completed to rank $n-1$ or less.
	Assume there exists a completion $M(\bfx)$ of $M$ such that $\rank(M(\bfx)_{1n,1n}) = n-3$.
	Let $\mathbf{y}(t)$ be obtained from $\bfx$ by perturbing each entry in the first and $n$th row
	and replacing the $(1,n)$ entry with the indeterminate $t$.
	Then $M(\bfy(t))$ is a partial matrix whose only unknown entry is $(1,n)$
	and $\det(M(\bfy(t))_{1n,1n}) = 0$.
	Since a generic row (respectively column) vector of size $n-2$ will not lie in the row span (column span) of $M(\bfx)_{1n,1n}$,
	$\det(M(\bfy(0))_{1,n}) \neq 0$.
	Lemma~\ref{lemma:oneMissingEdge} then implies that $M(\bfy(t))$, and therefore $M$, has a completion to rank $n-1$.
	If there exists a completion $M(\bfx)$ of $M$ such that $\rank(M(\bfx)_{1n,1n})=n-4$, then there exists $\bfx'$ obtained from $\bfx$
	via a generic perturbation of a single entry such that $\rank(M(\bfx')_{1n,1n})=n-3$.
	If there exists a completion $M(\bfx)$ such that $\rank(M(\bfx)_{1n,1n})\le n-5$, then $\rank(M(\bfx)) \le n-1$.
\end{proof}
Given full-rank typical $G$ and a $G$-partial matrix $M$,
Lemma \ref{lem:recursionToTestFullRank} gives us a recursive procedure for
determining whether or not $M$ must be completed to full rank.
We will use the following definition to convert this recursive procedure
into one where we just check the signs of various minors of $M(0)$.

\begin{defn}\label{defn:minorPoset}
    Let $G = (V,E)$ be a full-rank typical graph.
    Let
    \[
    	\mathcal{O} := (\{i_1,j_1\},\dots,\{i_k,j_k\})
    \]
    be an ordering of the non-edges of $G$.
    Initialize $\pi(G,\mathcal{O}):=\{V\}$.
    Iteratively for $l=1,\dots,k$
    and for each inclusion-wise minimal element $S$ of $\pi(G,\mathcal{O})$ such that $\{i_l,j_l\} \subseteq S$,
    add $S\setminus\{i_l\}$ and $S\setminus\{j_l\}$ to $\pi(G,\mathcal{O})$.
    Note that if we partially order $\pi(G,\mathcal{O})$ by inclusion,
    then every non-minimal $S \in \pi$ covers exactly two elements.
\end{defn}

Given an $n\times n$ symmetric (partial) matrix $M$ and $S \subseteq [n]$,
let $A_S$ denote the principal submatrix of $A$ with rows and columns indexed by $S$.

\begin{thm}\label{thm:whenFullRank}
    Let $G = (V,E)$ be full-rank typical, let $M$ be a generic $G$-partial matrix,
    and let $M(\bfx_0)$ be a generic completion of $M$.
    Let $\mathcal{O}$ be an ordering of the non-edges of $G$.
    Then $M(\bfx)$ is full-rank for all real $\bfx$ if and only if
    whenever $S_1,S_2$ are the elements covered by some $S \in \pi(G,\mathcal{O})$,
    $\sign(M(\bfx_0)_{S_1})$ and $\sign(M(\bfx_0)_{S_2})$ are nonzero and opposite.
\end{thm}
\begin{proof}
    Let $S$ be a non-minimal element of $\pi(G,\mathcal{O})$ and let
    $S_1$ and $S_2$ be the elements covered by $S$.
    If $\sign(M_{S_1}(\bfx_0)) = \sign(M_{S_2}(\bfx_0))$,
    then Lemma~\ref{lem:recursionToTestFullRank} implies that $M_S$
    has a completion to less than full rank, i.e.~that $\sign(M_S) = 0$.
    By Lemma~\ref{lem:recursionToTestFullRank}, this implies that $M$ has a completion to less than full rank.

    Now assume $\sign(M_{S_1}(\bfx_0))$ and $\sign(M_{S_2}(\bfx_0))$ are nonzero and opposite whenever $S_1$ and $S_2$
    are the elements covered by some $S \in \pi(G,\mathcal{O})$.
    Whenever $T \in \pi(G,\mathcal{O})$ is minimal, $M_T$ is fully-specified.
    Lemma~\ref{lem:recursionToTestFullRank} therefore implies
    that whenever $S$ covers two minimal elements of $\pi(G,\mathcal{O})$,
    all completions of $M_S$ are full-rank.
    So in this case, Proposition~\ref{prop:full-rankInertia} implies $\sign(M_S) = \sign(M_S(\bfx_0))$.
    It then follows by Lemma~\ref{lem:recursionToTestFullRank} and induction that
    $\sign(M_S) = \sign(M_S(\bfx_0))$ for all $S \in \pi(G,\mathcal{O})$.
    In particular, $\sign(M) = \sign(M_{[n]}) = \sign(\det(M(\bfx_0)))$ is nonzero, i.e.~all completions of $M$
    have rank $n$.
\end{proof}

\begin{ex}
	Figure \ref{fig:posetStructures} shows three examples of $\pi(G,\mathcal{O})$
	alongside the complement graph $G^c$.
	In all cases, $\mathcal{O}$ is the lexicographic ordering of the non-edges of $G$ (i.e.~the edges of $G^c$).
    	\begin{center}
    	\begin{figure}
    		\begin{tikzpicture}[roundnode/.style={circle, fill=black, inner sep=0pt, minimum size=1mm},xscale=.7,yscale=.8]
    		\tikzstyle{every node}=[font=\footnotesize]
    		\node at (-.5,2) {$G^c=$};
    		\node[roundnode,label=left:{\small $1$}](A) at (1,2.6) {};
    		\node[roundnode,label=left:{\small $2$}](A) at (1,2) {};
    		\node[roundnode,label=left:{\small $3$}](C) at (1,1.4) {};
    		\node[roundnode,label=right:{\small $4$}](A) at (2,2.6) {};
			\node[roundnode,label=right:{\small $5$}](B) at (2,2) {};
			\node[roundnode,label=right:{\small $6$}](C) at (2,1.4) {};
   		    		
    		\draw[-] (1,2.6) -- (2,2.6);
    		\draw[-] (1,2.6) -- (2,2);
    		\draw[-] (1,2.6) -- (2,1.4);
    		\draw[-] (1,2) -- (2,2.6);
			\draw[-] (1,2) -- (2,2);
			\draw[-] (1,2) -- (2,1.4);
    		\draw[-] (1,1.4) -- (2,2.6);
			\draw[-] (1,1.4) -- (2,2);
			\draw[-] (1,1.4) -- (2,1.4);

			\node at (8,3) {$123456$};
			\node at (7,2) {$12356$};
			\node at (9,2) {$23456$};
			\node at (6,1) {$1236$};
			\node at (8,1) {$2356$};
			\node at (10,1) {$3456$};
			\node at (5,0) {$123$};
			\node at (7,0) {$236$};
			\node at (9,0) {$356$};
			\node at (11,0) {$456$};
			\node at (6,-1) {$23$};
			\node at (8,-1) {$36$};
			\node at (10,-1) {$56$};
			\node at (7,-2) {$3$};
			\node at (9,-2) {$6$};

    		\draw[-] (7.7,2.7) -- (7.3,2.3);
    		\draw[-] (8.3,2.7) -- (8.7,2.3);
    		\draw[-] (6.7,1.7) -- (6.3,1.3);
    		\draw[-] (7.3,1.7) -- (7.7,1.3);
    		\draw[-] (8.7,1.7) -- (8.3,1.3);
    		\draw[-] (9.3,1.7) -- (9.7,1.3);
    		\draw[-] (5.7,.7) -- (5.3,.3);
    		\draw[-] (6.3,.7) -- (6.7,.3);
    		\draw[-] (7.7,.7) -- (7.3,.3);
    		\draw[-] (8.3,.7) -- (8.7,.3);
    		\draw[-] (9.7,.7) -- (9.3,.3);
    		\draw[-] (10.3,.7) -- (10.7,.3);
    		\draw[-] (6.3,-.7) -- (6.7,-.3);
    		\draw[-] (7.3,-.3) -- (7.7,-.7);
    		\draw[-] (8.3,-.7) -- (8.7,-.3);
    		\draw[-] (9.3,-.3) -- (9.7,-.7);
    		\draw[-] (7.3,-1.7) -- (7.7,-1.3);
    		\draw[-] (8.3,-1.3) -- (8.7,-1.7);

\node at (-6.5,-4) {$G^c=$};
\node[roundnode,label=left:{\small $1$}](A) at (-5,-3.4) {};
\node[roundnode,label=left:{\small $2$}](A) at (-5,-4) {};
\node[roundnode,label=left:{\small $3$}](C) at (-5,-4.6) {};
\node[roundnode,label=right:{\small $6$}](A) at (-4,-3.4) {};
\node[roundnode,label=right:{\small $5$}](B) at (-4,-4) {};
\node[roundnode,label=right:{\small $4$}](C) at (-4,-4.6) {};

\draw[-] (-5,-3.4) -- (-4,-3.4);
\draw[-] (-5,-4) -- (-4,-4);
\draw[-] (-5,-4.6) -- (-4,-4.6);

\node at (1,-3) {$123456$};
\node at (-1,-4) {$12345$};
\node at (3,-4) {$23456$};
\node at (-2,-5.3) {$1234$};
\node at (0,-5.3) {$1345$};
\node at (2,-5.3) {$2346$};
\node at (4,-5.3) {$3456$};
\node at (-2.5,-6.6) {$123$};
\node at (-1.5,-6.6) {$124$};
\node at (-.5,-6.6) {$135$};
\node at (.5,-6.6) {$145$};
\node at (1.5,-6.6) {$236$};
\node at (2.5,-6.6) {$246$};
\node at (3.5,-6.6) {$356$};
\node at (4.5,-6.6) {$456$};

\draw[-] (.7,-3.3) -- (-.7,-3.7);
\draw[-] (1.3,-3.3) -- (2.7,-3.7);
\draw[-] (-1.3,-4.3) -- (-1.7,-5);
\draw[-] (-.7,-4.3) -- (-.3,-5);
\draw[-] (2.7,-4.3) -- (2.3,-5);
\draw[-] (3.3,-4.3) -- (3.7,-5);
\draw[-] (-2.2,-5.6) -- (-2.3,-6.3);
\draw[-] (-1.8,-5.6) -- (-1.7,-6.3);
\draw[-] (-.2,-5.6) -- (-.3,-6.3);
\draw[-] (.2,-5.6) -- (.3,-6.3);
\draw[-] (1.8,-5.6) -- (1.7,-6.3);
\draw[-] (2.2,-5.6) -- (2.3,-6.3);
\draw[-] (3.8,-5.6) -- (3.7,-6.3);
\draw[-] (4.2,-5.6) -- (4.3,-6.3);

\node at (8.5,-4) {$G^c=$};
\node[roundnode,label=left:{\small $1$}](A) at (10,-4) {};
\node[roundnode,label=right:{\small $2$}](A) at (11,-3.4) {};
\node[roundnode,label=right:{\small $3$}](B) at (11,-4) {};
\node[roundnode,label=right:{\small $4$}](C) at (11,-4.6) {};

\draw[-] (10,-4) -- (11,-3.4);
\draw[-] (10,-4) -- (11,-4);
\draw[-] (10,-4) -- (11,-4.6);

\node at (15,-3) {$1234$};
\node at (14,-4) {$134$};
\node at (16,-4) {$234$};
\node at (13,-5) {$14$};
\node at (15,-5) {$34$};
\node at (12,-6) {$1$};
\node at (14,-6) {$4$};

\draw[-] (14.7,-3.3) -- (14.3,-3.7);
\draw[-] (15.3,-3.3) -- (15.7,-3.7);
\draw[-] (13.7,-4.3) -- (13.3,-4.7);
\draw[-] (14.3,-4.3) -- (14.7,-4.7);
\draw[-] (12.7,-5.3) -- (12.3,-5.7);
\draw[-] (13.3,-5.3) -- (13.7,-5.7);

\end{tikzpicture}
    		\caption{Various examples of $\pi(G,\mathcal{O})$ alongside $G^c$. In all cases $\mathcal{O}$ is
    		the lexicographic ordering of the non-edges.} 		
    		\label{fig:posetStructures}
    	\end{figure}
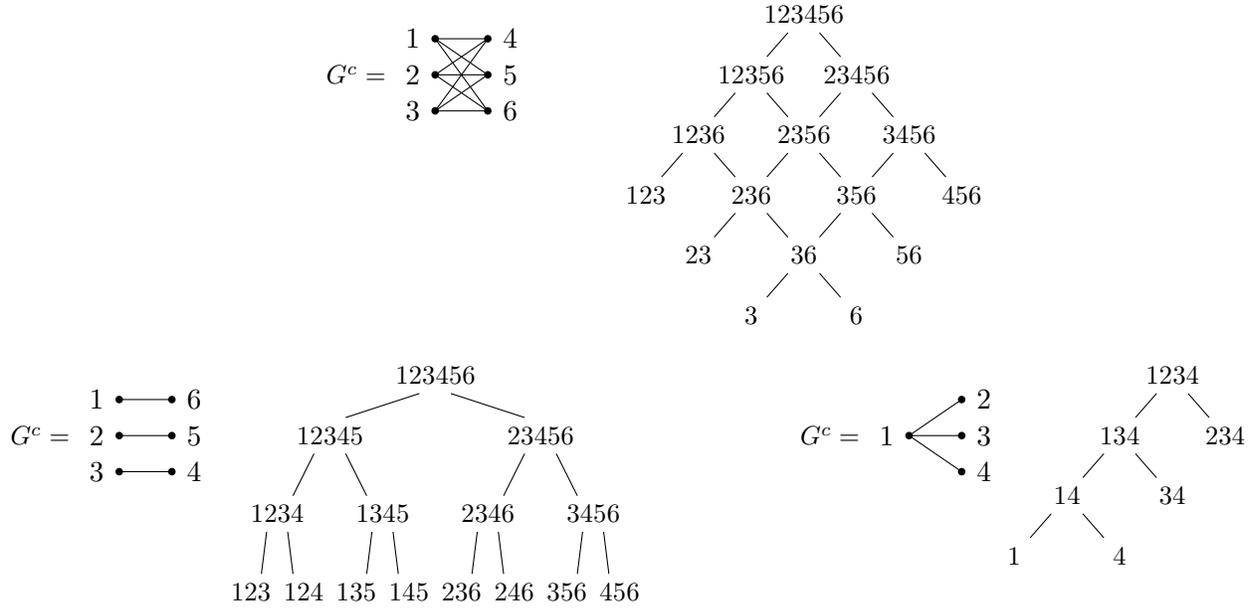
    \end{center}
    
\end{ex}

\subsection{Disjoint unions of full-rank typical graphs}
In this subsection, we study the typical ranks of disjoint unions of full-rank typical graphs.
Proposition \ref{prop:full-rankInertia} allows us to make the following definition.

\begin{defn}\label{defn:esdFullRankTypicalGraphs}
	Let $G$ and $H$ be full-rank typical graphs,
	and let $M$ and $N$ respectively be $G$- and $H$-partial matrices which are minimally completable to full rank.
	Define $\esd(M,N)$ to be $\esd(M(\bfx),N(\bfx))$ for any completion $A(\bfx),B(\bfx)$ of $A$ and $B$.
\end{defn}

\begin{prop}\label{prop:esdArgument-fullranktypical}
	Let $G_1$ and $G_2$ be full-rank typical graphs
	and define $G:= G_1\sqcup G_2$.
	Let $M$ be a $G$-partial matrix, which we may write as
	\[
		M=\begin{pmatrix}
		M_1 & ?\\
		? & M_2
		\end{pmatrix}
	\]
	where $M_i$ is a $G_i$-partial matrix.
	If $M_1$ and $M_2$ are minimally completable to full rank,
	then $M$ is minimally completable to rank $\max\{n,m\}+\esd(M_1,M_2)$.
\end{prop}
\begin{proof}
	This follows immediately from Proposition \ref{prop:full-rankInertia} and Theorem \ref{thm:esdArgument-disjointcliques}.
\end{proof}

We end this section with a characterization of the maximum typical ranks
of disjoint unions of more than two cliques.

\begin{prop}\label{prop:disjointThreeCliques}
	Let $G$ be the disjoint union of $k$ cliques
	where the $i^{\rm th}$ clique has size $n_i$
	and $n_1 \ge n_2 \ge \dots \ge n_k$.
	Then the maximum typical rank of $G$ is $n_1 + n_2$.
\end{prop}
\begin{proof}
	Corollary \ref{cor:maxTypicalLowerBound} implies that the maximum typical rank of $G$
	is at least $n_1 + n_2$.
	To prove the other direction, let $M(\bfx)$ be a generic $G$-partial matrix.
	We may write
	\[
	M(\bfx)=\begin{pmatrix}
	A(\bfx) & Y(\bfx)\\
	Y(\bfx)^T & B(\bfx)
	\end{pmatrix}.
	\]
	where
	\[
	A(\bfx):=\begin{pmatrix}
	M_1 & X_{12}\\
	X_{12}^T & M_2
	\end{pmatrix}
	\qquad
	B(\bfx):=\begin{pmatrix}
	M_3 & X_{34} & \cdots & X_{3k} \\
	X_{34}^T & M_4 & \cdots & X_{4k} \\
	\vdots & \vdots & \ddots & \vdots \\
	X_{3k}^T & X_{4k}^T & \cdots & M_k
	\end{pmatrix}
	\qquad
	Y(\bfx) :=
	\begin{pmatrix}
	    X_{13} & \cdots & X_{1k}\\
	    X_{23} & \cdots & X_{2k}
	\end{pmatrix}.
	\]
	where each $M_i$ is a $n_i\times n_i$ fully specified symmetric matrix
	and each $X_{ij}$ is a matrix of indeterminates.
	Just as in the proof of Theorem \ref{thm:esdArgument-disjointcliques},
	we may use the eigendecomposition of each $M_i$ to obtain a linear change of variables
	so that each $M_i$ is diagonal.
	Therefore, without loss of generality, assume $M_i=\diag(m_{i1},\dots,m_{in_i})$ for all $i=1,\dots,k$.
	
	First, consider the case that $\esd(M_1,M_2)=n_2$.
	Without loss of generality, assume that $M_1$ is positive definite and $M_2$ is negative definite.
	We now describe a completion of $M$ to rank $n_1 + n_2$,
	similar to the construction given in the proof of Theorem \ref{thm:esdArgument-disjointcliques}.
	Denote $(p,q)$-entry of $X_{ij}$ by $(X_{ij})_{pq}$.
	For all $X_{ij}$ that are not blocks of $Y$ (i.e.~$i \ge 3$, or $i = 1$ and $j = 2$),
	set $X_{ij} = 0$.
	For the $X_{ij}$ that are blocks of $Y$ (i.e.~$i = 1,2$ and $j \ge 3$)
	set $(X_{ij})_{pq} = 0$ when $p \neq q$,
	and specify the remaining entries as follows
	\begin{eqnarray*}
		(X_{1j})_{pp}=\left\{\begin{array}{cl}
			\sqrt{\frac{m_{jp}}{m_{1p}}} & \text{if } \ m_{jp}>0\\
			0 & \text{otherwise}
		\end{array}\right.
		\qquad
		(X_{2j})_{pp}=\left\{\begin{array}{cl}
			\sqrt{\frac{m_{jp}}{m_{2p}}} & \text{if } \ m_{jp}<0\\
			0 & \text{otherwise}.
		\end{array}\right.
	\end{eqnarray*}
	This ensures that $B(0)-Y^T A(0)^{-1} Y$ is a zero matrix.
	The rank formula for Schur complements then gives
	\[
	\rank(M(\bfx))=\rank(A(0))+\rank\left(B(0)-Y^T A(0)^{-1} Y\right)=n_1+n_2.
	\]
	
	Now, consider the case that $\esd(M_1,M_2)<n_2$.
	If $\esd(M_1,M_2)<\esd(M_1,M_i)$ for some $i>2$,
	since $n_1 \ge \dots \ge n_k$,
	we may proceed by relabeling the blocks so that $\esd(M_1,M_2)=\max\limits_{i=2,\dots,k}\esd(M_1,M_i)$
	and exhibiting a completion of $M(\bfx)$ to rank $n_1 + n_2$.
	Since $i = 2$ maximizes $\esd(M_1,M_i)$,
	after possibly re-ordering rows and columns,
	we may assume that for any diagonal entry $m_{il}$,
	either $m_{il}m_{1l}>0$ or $m_{il}m_{2l}>0$.
	Hence, we can complete $X_{1i}$ and $X_{2i}$ as before
	to ensure that $B(0)-Y^T A(0)^{-1} Y$ is a zero matrix.
	Since $\rank(A(0)) = n_1 + n_2$,
	the rank formula for Schur complements implies
	\[
	\rank(M(\bfx))=\rank(A(0))+\rank\left(B(0)-Y^T A(0)^{-1} Y\right)=n_1+n_2. \qedhere
	\]
\end{proof}

\section{Low maximum typical ranks}

\subsection{Generic completion rank one}
We first look at the semisimple graphs which have generic completion rank one.
There is a known characterization for such graphs.

\begin{prop}[{\cite[Proposition 5.3]{singer2010uniqueness}}]\label{prop:gcr1graphs}
	The generic completion rank of a semisimple graph $G$ is one if and only if
	$G$ is free of even cycles,
	and every connected component of $G$ has at most one odd cycle.
\end{prop}

Proposition \ref{prop:typicalrankProperties} implies that if $G$ has generic completion rank one,
then the maximum typical rank of $G$ is at most two.
The following proposition tells us when two is attained. 

\begin{thm}\label{thm:typicalrank2Graphs}
	Let $G$ be a semisimple graph with generic completion rank $1$.
	Then $G$ has $2$ as a typical rank if and only if $G$ has at least two odd cycles.
\end{thm}
\begin{proof}
	We begin by showing that if $G$ has two odd cycles, then $G$ has two as a typical rank.
	If each of these cycles is a loop, then $G$ has $K_1^\circ \sqcup K_1^\circ$
	as an induced subgraph.
	Setting the two corresponding diagonal entries to values with opposite signs
	yields a principal $2 \times 2$ minor that will have a strictly negative determinant for any completion.
	If one of these odd cycles has size three,
	then any $G$-partial matrix $M(\bfx)$ has a sub partial matrix $N$ of the form
	\begin{equation}\label{eq:3cycle}
	N(\bfx) = 
	\begin{pmatrix}
	x_{1} & a & b \\
	a & x_{2} & c \\
	b & c & x_{3}
	\end{pmatrix}.
	\end{equation}
	Thus any rank-one completion of $M(\bfx)$ must set $x_2 = \frac{ac}{b}$.
	If one of the cycles has size greater than three,
	choose three edges in this odd cycle that form a path.
	We can write the corresponding sub partial matrix $N$ of $M$ as
	\begin{equation}\label{eq:3path}
	N(\bfx) = 
	\begin{pmatrix}
	x_0 & a & x_1 & x_2 \\
	a & x_3 & b & x_4 \\
	x_1 & b & x_5 & c \\
	x_2 & x_4 & c & x_6
	\end{pmatrix}.
	\end{equation}
	Thus any rank-one completion of $M(\bfx)$ must set $x_2 = \frac{ac}{b}$.
	Let $G'$ be the graph obtained from $G$ by deleting two of the vertices in a three-cycle
	and adding a loop onto the remaining vertex,
	or by deleting the inner vertices in a path of length three
	and adding an edge between the two outer edges.
	In both cases, let $M'(\bfx)$ be the $G'$-partial matrix
	obtained from $M(\bfx)$ by deleting the corresponding rows and columns,
	and setting $x_2$ to $\frac{ac}{b}$.
	Note that $G'$ has generic completion rank one,
	and that if $M'(\bfx)$ is minimally completable to rank $2$, then so is $M(\bfx)$.
	As $a,b$, and $c$ range over $\mathbb{R}$, $\frac{ac}{b}$ takes all real values.
	So if $G'$ has two as a typical rank, then for some choice of $\bfx$,
	$M'(\bfx)$ is minimally completable to rank two.
	That $G'$ has two as a typical rank now follows by induction.

	We now show that if $G$ has at most one odd cycle,
	then $G$ only has $1$ as a typical rank.
	Let $M(\bfx)$ be a $G$-partial matrix.
	We begin by assuming that $G$ is itself an odd cycle.
	The case that $G$ is a $1$-cycle is trivial,
	and if $G$ is a $3$-cycle, we have formulas for the diagonal entries of $M(\bfx)$ as in \eqref{eq:3cycle}.
	If $G$ has length greater than $3$, then we have formulas for all the off-diagonal entries of $M(\bfx)$
	as in \eqref{eq:3path},
	then formulas for the diagonal entries as in \eqref{eq:3cycle}.
	Now assume that $G$ is connected.
	If $G$ has no odd cycle, add a loop and set the corresponding diagonal entry of $M(\bfx)$ to a generic real number.
	Otherwise, construct a rank-one completion of $M(\bfx)$ corresponding to the odd cycle as before,
	and add the corresponding edges to $G$.
	Now all the remaining unknown diagonal entries can be completed to the same sign.
	To see this, note that if $G$ has an induced subgraph
	consisting an edge joining a 
	vertex having a loop to a vertex with no loop,
	$M(\bfx)$ has a sub partial matrix $N(\bfx)$ of the form
	\[
	N(\bfx) = \begin{pmatrix}
	a & b \\
	b & x
	\end{pmatrix}
	\]
	where $x$ is unknown.
	In a rank-1 completion, we must have $x = \frac{b^2}{a}$, which will have the same sign as $a$.
	Thus when we complete the missing diagonals according to this formula, they will all have the same sign.
	Then, we can complete the off diagonal entries over the reals.
	
	If $G$ is disconnected, then at most one connected component has an odd cycle.
	Complete the missing entries of $M(\bfx)$ in this component with the odd cycle as in the previous case.
	On each remaining component, set one of the diagonal entries to a generic real number whose sign
	is the same as the sign of the diagonal entries.
	Then we can complete each connected component as before,
	and since the signs of all the diagonal entries will be equal,
	we can then complete the remaining unknown off-diagonal entries.
\end{proof}

\subsection{Looped graphs}
A vertex $v$ of a graph $G$ is said to be \emph{looped} if $G$ has a loop at $v$.
We say that a graph $G$ is \emph{looped} if every vertex of $G$ is looped.
The main result of this subsection is Theorem \ref{thm:loopedTrees} below.
It gives a combinatorial characterization of the typical ranks of a looped graph that has generic completion rank at most $2$.
Recall that a \emph{star tree} is a tree with at most one non-leaf vertex.

\begin{thm}\label{thm:loopedTrees}
    Let $G$ be a looped graph.
    \begin{enumerate}
    	\item\label{item:gcrAtMostTwo} The generic completion rank of $G$ is at most $2$ if and only if $G$ is a looped forest.
    	Equality is attained if and only if $G$ has at least one non-loop edge.
    	\item\label{item:gcrOneLooped} When $G$ has generic completion rank $1$ and at least two vertices, $G$ also has $2$ as a typical rank.
    	\item\label{item:gcrTwoLooped} When $G$ has generic completion rank $2$,
    	the maximum typical rank of $G$ is
    	\begin{enumerate}
    	    \item $2$ if $G$ has exactly two vertices
    	    \item $3$ if $G$ has at least three vertices 
    	    and is the union of a looped star tree and a looped set of isolated vertices, and
    	    \item $4$ otherwise.
    	\end{enumerate}
    \end{enumerate}
\end{thm}

Before proving Theorem \ref{thm:loopedTrees},
we give a handful of intermediate results.
Recall that a \emph{suspension vertex} in a graph is a vertex that is adjacent to every other vertex.
The relevance of the following lemma to Theorem \ref{thm:loopedTrees} comes from the fact
that a star tree can be obtained from a set of isolated vertices by adding a suspension vertex.

\begin{lem}\label{lem:addingSuspensionVertex}
	Adding a looped suspension vertex to a semisimple graph increases all its typical ranks by 1.
\end{lem}
\begin{proof}
	Let $G = ([n],E)$ be a semisimple graph and let $r$ be a typical rank of $G$.
	Let $H = ([n+1],E')$ be obtained from $G$ by adding a looped suspension vertex and let $N(\bfx)$ be an $H$-partial matrix.
	We may assume that $N(\bfx)$ is of the following form
	\[
		N(\bfx)=\begin{pmatrix}
		\alpha & v \\
		v^T & M(\bfx)
		\end{pmatrix}
	\]
	where $M(\bfx)$ is a $G$-partial matrix,
	$v$ is a fully specified row vector and $\alpha$ is a nonzero real number.
	Given a completion $\bfx_0$, the rank formula for Schur complements gives
	\[
		\rank(N(\bfx_0))=1+\rank\left(M(\bfx_0)-\frac{1}{\alpha}v^Tv\right).
	\]
	Applying a linear change of coordinates on the space of $G$-partial matrices
	and a linear change of variables for the substitution,
	it is evident that we may choose a generic $M$ such that
	the minimum of $\rank\left(M(\bfx_0)-\frac{1}{\alpha}v^Tv\right)$ is $r$.
	Thus $H$ has $r+1$ as a typical rank.
\end{proof}

The characterization of looped graphs with generic completion rank at most two, given below,
is an easy consequence of a result of Gross and Sullivant.

\begin{prop}\label{prop:gcrOfAllLoopedTrees}
	Let $G$ be a looped graph.
	The generic completion rank of $G$ is at most two if and only if $G$ is a looped forest,
	with equality attained if and only if $G$ has at least one non-loop edge.
\end{prop}
\begin{proof}
    This follows from \cite[Theorem 2.5]{GSMR3706762}.
\end{proof}

Proposition \ref{prop:typicalrankProperties} implies
that the maximum typical rank of a graph with generic completion rank $2$,
which in the looped case we now know to be trees,
is at most $4$.
Lemma \ref{lemma:startree} below tells us that for looped star trees,
this inequality is strict.

\begin{lem}\label{lemma:startree}
	If $G$ is a looped star tree with at least three vertices, then the typical ranks of $G$ are $2$ and $3$.
\end{lem}
\begin{proof}
	By Proposition \ref{prop:gcr1graphs} and Theorem \ref{thm:typicalrank2Graphs}, we have that the generic completion rank of a union of isolated looped vertices is $1$ and its maximum typical rank is $2$.
	The proposition then follows from Lemma \ref{lem:addingSuspensionVertex}.
\end{proof}

Proposition \ref{prop:maximumtypicalrank3graphs} below handles most of the heavy lifting
in the proof of Theorem \ref{thm:loopedTrees}.
It tells us exactly which looped graphs have $3$ as their maximum typical rank.
Before we can prove that,
we need Proposition \ref{prop:upperBoundForMaxTypicalRank}
which gives an upper bound on the maximum typical rank of a graph
in terms of the maximum size of an independent set of vertices.

\begin{prop}\label{prop:upperBoundForMaxTypicalRank}
	Let $G$ be a looped semisimple graph with $n$ vertices. 
	Let $r$ be the maximum size of an independent set of vertices of $G$.
	Then, the maximum typical rank of $G$ is at most $2+n-r$.
\end{prop}
\begin{proof}
	Let $H$ be the graph obtained from $G$ by removing all vertices not in a particular independent set of vertices of size $r$.
	Proposition \ref{prop:gcrOfAllLoopedTrees} implies that the maximum typical rank of $H$ is at most $2$.
	Let $H'$ be obtained from $H$ by adding $n-r$ looped suspension vertices.
	Lemma \ref{lem:addingSuspensionVertex} implies that the maximum typical rank of $H'$ is at most $2+n-r$.
	Since $G$ is a subgraph of $H'$,
	the maximum typical rank of $G$ is also at most $2+n-r$.
\end{proof}

\begin{prop}\label{prop:maximumtypicalrank3graphs}
	Let $G$ be a looped graph with at least three vertices and one non-loop edge.
	Then the maximum typical rank of $G$ is $3$ if and only if $G$ is a looped triangle,
	or the disjoint union of a looped star tree and a (possibly empty) set of looped isolated vertices.
\end{prop}
\begin{proof}
	Since adding edges and vertices to a graph can only increase its maximum typical rank,
	Proposition \ref{prop:KnKmfullrank} implies that if $G$ has a pair of edges not sharing any vertex,
	then $G$ has four as a typical rank.
	Thus if the maximum typical rank of $G$ is $3$,
	then $G$ is either a triangle, or the disjoint union of a star tree and a (possibly empty) set of isolated vertices.
	It is clear that a triangle has $3$ as a typical rank.
	In the other case, the fact that $G$ has $3$ as a typical rank is implied by 
	Lemma \ref{lemma:startree} when the star tree in $G$ has at least three vertices,
	and by Proposition \ref{prop:disjointThreeCliques} when the star tree in $G$ has two vertices.
	
	The proposition now follows by noting that the triangle cannot have four as a typical rank (it only has three vertices),
	and Proposition \ref{prop:upperBoundForMaxTypicalRank} implies that neither can the disjoint union
	of a star tree and a set of isolated vertices.
\end{proof}

We are now ready to prove the main result of this subsection.

\begin{proof}[Proof of Theorem \ref{thm:loopedTrees}]
	Proposition \ref{prop:gcrOfAllLoopedTrees} is \eqref{item:gcrAtMostTwo}.
	Theorem \ref{thm:fullranktypicalCharacterization} implies \eqref{item:gcrOneLooped}.
	If $G$ has generic completion rank two, then Proposition \ref{prop:typicalrankProperties}
	implies that the maximum typical rank of $G$ is at most $4$.
	In this case, it is clear that if $G$ has two vertices, then $2$ is the maximum typical rank of $G$.
	So assume that $G$ has at least $3$ vertices.
	If $G$ is the union of a looped star tree and a set of looped isolated vertices,
	then Proposition \ref{prop:maximumtypicalrank3graphs} implies that the maximum typical rank of $G$ is $3$.
	In all other cases, $G$ has $K_2^\circ \sqcup K_2^\circ$ as a subgraph.
	Any graph obtained by adding a (possibly empty) set of edges to this subgraph
	has typical rank $4$ by Theorem \ref{thm:fullranktypicalCharacterization}.
	In this case, $G$ has $4$ as its maximum typical rank.
\end{proof}

\section{Open problems}

In this section, we list several open problems that seem like promising next steps for the study of typical
ranks of semisimple graphs.
Theorem \ref{thm:loopedTrees} gives us a relatively complete understanding of the typical rank behavior
of cycle-free looped forests,
so a natural next step is to investigate the typical rank behavior of cycles.

We know that every looped cycle has generic completion rank $3$ \cite[Theorem 2.5]{GSMR3706762}.
Moreover, if $G$ is a looped cycle of length 4 or greater, then Theorem \ref{thm:fullranktypicalCharacterization}
implies that $G$ also has $4$ as a typical rank
since $G$ contains $K_2^\circ\sqcup K_2^\circ$ as a subgraph.
Theorem \ref{thm:loopedTrees} implies that every path has a maximum typical rank of at most $4$,
so since a cycle can be obtained from a path by adding a single suspension vertex
and deleting edges, Lemma \ref{lem:addingSuspensionVertex} implies that no looped cycle can have $6$ as a typical rank.
However, at this point, we do not know whether $5$ is a possible typical rank.
Therefore we ask the following question.

\begin{ques}
    Does there exist a looped cycle with $5$ as its maximum typical rank?
\end{ques}


There is also much left to be done in understanding generic completion ranks.
In particular, we do not even know of a characterization of the semisimple graphs with generic completion rank two.
We therefore also pose the following question,
whose answer is known for looped graphs \cite{GSMR3706762} 
and bipartite graphs \cite{bernstein2017completion}.

\begin{ques}
	Which semisimple graphs have $2$ as their generic completion rank?
\end{ques}

Given looped graphs $G$ and $H$,
the generic completion rank of any clique sum of $G$ and $H$ is
the maximum of the generic completion ranks of $G$ and $H$ \cite[Theorem 1.12]{MR3903791}.
Proposition \ref{prop:esdArgument-fullranktypical} gives us some information about how typical rank behaves
in the context of disjoint unions of full rank typical graphs.
In light of this, we ask the following more general question.

\begin{ques}\label{ques:cliqueSums}
    What are the typical rank of a clique sums of two semisimple graphs?
\end{ques}

The following proposition gives yet another case study for the disjoint union of graphs.
Its proof motivates Question \ref{ques:inertia} which follows.

\begin{prop}\label{prop:two4Cycles}
	The maximum typical rank of the disjoint union of two looped $4$-cycles $C_4^\circ\sqcup C_4^\circ$ is~$4$.	
\end{prop}
\begin{proof}
	Let $M(x,y)$ be a $C_4^\circ$-partial matrix, which we write as
	\begin{equation*}
	M(x,y)=\left(\begin{array}{c|c}
	A & \begin{array}{cc}
	x & \mu \\
	\lambda & y	
	\end{array}
	\\
	\hline
	\begin{array}{cc}
	x & \lambda \\
	\mu & y	
	\end{array} & B
	\end{array}\right)	
	\end{equation*}
	where $A=(a_{ij})$ and $B=(b_{ij})$ are $2\times 2$ fully specified matrices,
	$\lambda, \mu$ are specified entries and $x,y$ are the unspecified entries.
	Theorem \ref{thm:fullranktypicalCharacterization} implies that $C_4^\circ$ is full-rank typical.
	We will show that if $M$ is minimally completable to rank $4$,
	then $\In(M)=(2,2,0)$.
	The desired result then follows from Proposition \ref{prop:esdArgument-fullranktypical}.
	
	We claim that a minimum rank completion of $M$ cannot be definite.
	For the sake of contraction,
	assume without loss of generality that a minimum rank completion of $M$ is positive definite.
	Then, for any $x,y$, the third leading principal minor of $M(x,y)$ must be positive.
	This minor is a quadratic polynomial in $x$ with leading coefficient $-a_{22}$.
	Since $M(x,y)$ is positive definite, $a_{22}>0$.
	But then for large $x$, the third leading principal minor becomes negative,
	thus contradicting that $M(x,y)$ is positive definite.
	
	Now, for the sake of contradiction, assume $M$ is minimally completable to full rank,
	and that $\In(M)=(3,1,0)$.
	Then, $\det(M(x,y))$ is negative for any completion of $M$. 
	Note that $\det(M)$ is a degree $4$ polynomial which has leading term $x^2y^2$.
	Thus, for large $x$ and $y$, it is clear that $\det(M(x,y))> 0$.
	This implies that $\In(M(x,y))$ cannot be $(3,1,0)$ nor $(1,3,0)$
	since that would imply $\det(M(x,y))<0$.
	The only remaining possibility is $\In(M(x,y)) = (2,2,0)$.
\end{proof}

Recall from Proposition \ref{prop:full-rankInertia}
that if $G$ is full-rank typical and $M$ is a $G$-partial matrix that is minimally completable to full rank,
then all completions of $M$ have the same inertia,
which we denote $\In(M)$.
The proof of Proposition \ref{prop:two4Cycles} suggests that
for the purposes of determining the maximum typical rank of a disjoint union of full-rank typical graphs,
it could be helpful to characterize the possible values of
of $\In(M)$ as $M$ ranges over all $G$-partial matrices that are minimally completable to full rank.
Thus we pose the following question.

\begin{ques}\label{ques:inertia}
	Given a full-rank typical graph $G$,
	what are the possible values of $\In(M)$
	as $M$ ranges over all $G$-partial matrices that are minimally completable to full rank?
\end{ques}

Given a full-rank typical $G$,
if $G^c$ has a proper two-coloring with color classes of size $m$ and $n$,
then there exists a generic $G$-partial matrix $M$
that is minimally completable to full rank and has $\In(M) = (m,n,0)$.
To see this, note that in this case $G$ has $K_m^\circ \sqcup K_n^\circ$ as a subgraph,
and so Proposition \ref{prop:KnKmfullrank} implies that
if in a $G$-partial matrix $M$,
the entries corresponding to the edges of the $K_m^\circ$ form a positive definite matrix
and the entries corresponding to the edges of the $K_n^\circ$ form a negative definite matrix,
then $M$ is minimally completable to full rank and $\In(M) = (m,n,0)$.
Since we were unable to find a full rank typical graph $G$
and $G$-partial matrix $M$ whose inertia did not correspond to a two-coloring of $G^c$ in this way,
we make the following conjecture.

\begin{conj}
	Let $G$ be a full-rank typical graph.
	Then there exists a $G$-partial matrix $M$, minimally completable to full rank,
	such that $\In(M) = (m,n,0)$ if and only if there exists a proper bicoloring
	of $G^c$ with $m$ red vertices and $n$ blue vertices.
\end{conj}

\section*{Acknowledgments}

We thank the NSF-supported (DMS-1439786) Institute for Computational and Experimental Research in Mathematics (ICERM) in Providence, Rhode Island, for supporting all three authors during the Fall 2018 semester.
We also thank Rainer Sinn for his help organizing the matrix completion working group at ICERM
where this paper began.
Daniel Irving Bernstein was supported by an NSF Mathematical Sciences Postdoctoral Research Fellowship (DMS-1802902).
Grigoriy Blekherman was partially supported by NSF grants DMS-1352073 and DMS-1901950.	
Kisun Lee was partially supported by NSF grant DMS-1719968. 
	
\bibliographystyle{plain}

\end{document}